\documentclass[12pt]{amsart}

\usepackage[dvipsnames, table]{xcolor}
\usepackage{tikz} 
\usepackage{tikz-cd} 
\usetikzlibrary{matrix,arrows,decorations.pathmorphing} 
\usepackage{courier} \usepackage{amssymb, amsmath, mathtools,
  amsfonts, amsthm} 
\usepackage[mathscr]{eucal} 
\usepackage{thmtools,
  thm-restate} 
\usepackage{pdfpages} 
\usepackage{graphicx} 
\usepackage{fullpage} 
\usepackage{caption} 
\usepackage{color}
\usepackage{latexsym} 
\usepackage{enumitem} 
\usepackage[toc,page]{appendix} 
\usepackage{multicol} 
\usepackage[colorlinks]{hyperref} 
\usepackage{url}
\hypersetup{colorlinks=true,urlcolor=blue,citecolor=blue,linkcolor=blue}
\usepackage[numbers]{natbib} 
\usepackage{setspace}  
\usepackage{indentfirst} 
\usepackage{listings} 
\usepackage{comment} 
\usepackage{manfnt} 
\usepackage{longtable} 
\usepackage{colonequals} 
\usepackage{nicefrac} 

\newcommand{\ZZ}{\mathbb{Z}} 
\newcommand{\QQ}{\mathbb{Q}} 
\newcommand{\CC}{\mathbb{C}} 

 \newcommand{\OO}{\mathscr{O}}
\newcommand{\Fq}{\mathbb{F}_q} 

\numberwithin{equation}{section} 
\newtheorem{thmx}{Theorem}

\newtheorem{theorem}{Theorem}[section]

\newtheorem{lemma}[theorem]{Lemma}
\newtheorem{coro}[theorem]{Corollary}

\theoremstyle{definition} \newtheorem{defn}[theorem]{Definition}

\newtheorem{remark}[theorem]{Remark}

\theoremstyle{remark}

\usepackage[breakable, theorems, skins]{tcolorbox} \tcbset{enhanced}

\newcounter{exercise}[subsection]

\newcommand{\md}{\text{ mod }} 
 
\newcommand{\mF}{\mathbb{F}}

\newcommand{\cC}{\mathcal{C}} 
 
\newcommand{\cS}{\mathcal{S}} 
\newcommand{\cO}{\mathcal{O}} 
 
\newcommand{\cQ}{\mathcal{Q}}

 \newcommand{\mand}{\text{ and }}

 \newcommand{\inv}{^{-1}}

    \newcommand{\mfor}{\text{ for }} \newcommand{\mif}{\text{ if
  }}

\newcommand{\brc}[1]{\left[ #1 \right]}

\newcommand{\unit}{^{\times}}

 \newcommand{\nrm}{\mathrm{N}}
 
 \newcommand{\paren}[1]{{\left(#1\right)}}

\def\quotient#1#2{\raise1ex\hbox{$#1$}{\Large/} \lower1ex\hbox{$#2$}}

\newcommand{\cdef}[1]{{\color{black}
    \textsf{\textsf{#1}}}} 

 \DeclareMathOperator{\Frob}{Frob}
 
\DeclareMathOperator{\GL}{GL} 
 
\DeclareMathOperator{\Tr}{Tr}

\DeclareMathOperator{\Gal}{Gal}

\title{Mertens' theorem for Chebotarev sets}

\author{Santiago Arango-Piñeros} \address{Department of Mathematics,
  Emory University, Atlanta, GA 30322, USA}
\email{santiago.arango@emory.edu}
\urladdr{\url{http://www.math.emory.edu/~sarang2/}}

\author{Daniel Keliher} \address{Department of Mathematics, Tufts
  University, Medford, MA 02144, USA} \email{daniel.keliher@tufts.edu}
\urladdr{\url{https://www.danielkeliher.com/}}

\author{Christopher Keyes} \address{Department of Mathematics, Emory
  University, Atlanta, GA 30322, USA}
\email{christopher.keyes@emory.edu}
\urladdr{\url{http://www.math.emory.edu/~ckeyes3/}}

\allowdisplaybreaks
\begin{document}

\maketitle

\begin{abstract}
  We generalize Mertens' product theorem to Chebotarev sets of prime ideals in Galois
  extensions of number fields. Using work of Rosen, we extend an argument of Williams from  cyclotomic extensions to this more general case. Additionally, we compute these products for Cheboratev sets in abelian extensions, $S_3$ sextic extensions, and sets of primes represented by some quadratic forms.  
\end{abstract}

\section{Introduction}
The Chebotarev density theorem is a deep generalization of the prime
number theorem; it contains Dirichlet's theorem for primes in
arithmetic progressions as a special case. In \cite{williams1974},
Williams proved Mertens' theorem for primes in arithmetic
progressions. Here, adapting Williams' method, we generalize Mertens' theorem to Chebotarev sets of prime ideals in a number field. Given a Galois extension of number fields $E/ F$ with Galois group $G\colonequals \Gal(E/F)$, and given a conjugacy class $C \subseteq G$, we prove
\begin{equation}
  \label{eq:main}
  \prod_{\substack{\nrm P \, \leq \, x \\
      \Frob_P = C}} \paren{1-\dfrac{1}{\nrm P}} \sim
  \paren{\dfrac{e^{-\gamma(E/F, C)}}{\log x}}^{|C|/|G|}, \quad \text{ as } x \to \infty ,
\end{equation}
where $P$ runs over all primes of $F$ which are unramified in $E$ and with absolute
norm bounded by $x$. In addition, we provide a power saving error term
and a description of the constant
$e^{-\gamma(E/F, C)}$.

Taking $E = F = \QQ$, \eqref{eq:main} specializes to Mertens' theorem \cite{mertens1874}; i.e.,
\begin{equation}
  \label{eq:Mertens}
  \prod_{p \, \leq \, x} \paren{1-\dfrac{1}{p}} \sim
  \dfrac{e^{-\gamma}}{\log x}, \quad \text{ as } x \to \infty,
\end{equation}
where $\gamma$ is the Euler constant. See, for example, \cite[Theorem 2.7]{montgomeryvaughn2006} for a modern discussion and proof of \eqref{eq:Mertens}.

Further, taking a cyclotomic extension $E = \QQ(\zeta_b) \supset F = \QQ$, the Galois group is isomorphic to $(\ZZ/b\ZZ)\unit$. Picking the conjugacy class corresponding to some element $a \in (\ZZ/b\ZZ)\unit$, and letting $\varphi(b) \colonequals \#(\ZZ/b\ZZ)\unit$ be the usual totient function, \eqref{eq:main} specializes to Williams' theorem \cite[Theorem 1]{williams1974}
\begin{equation}
  \label{eq:Williams}
  \displaystyle\prod_{\substack{p\, \leq \, x \\ p \, \equiv \, a \md b}}\left(1 - \dfrac{1}{p}\right) \sim \left(\dfrac{e^{-\gamma(a,b)}}{\log x}\right)^{1/\varphi(b)}, \quad \text{ as } x \to \infty.
\end{equation}

Our result relies heavily on work of Rosen \cite[Theorem 2]{rosen1999},
who proved the Mertens' analog of Landau's prime ideal
theorem. Taking $E = F \supset \QQ$,  \eqref{eq:main} specializes to Rosen's result; i.e.,
\begin{equation}
  \label{eq:rosen}
  \prod_{\nrm P \, \leq \, x } \paren{1-\dfrac{1}{\nrm P}} \sim
  \dfrac{e^{-\gamma_E}}{\log x}, \quad \text{ as } x \to \infty.
\end{equation}
We summarize these cases, in analogy with the corresponding prime number theorems, in the following table. 
\begin{center}
    \begin{table}[h]
    \setlength{\arrayrulewidth}{0.5mm}
    \setlength{\tabcolsep}{5pt}
    \renewcommand{\arraystretch}{2}
    \caption{Prime Number theorems vs. Mertens-type theorems}
        \rowcolors{1}{gray!25}{white!100}
        \begin{tabular}{|c|c|c|}
        \hline
        Trivial extension & Prime Number theorem & Mertens' theorem  \\
        $E = F = \mathbb{Q}$ & $\displaystyle \sum_{p \, \leq \, x} 1 \sim \tfrac{x}{\log x}$ & $\displaystyle\prod_{p \, \leq \, x}\left(1 - \tfrac{1}{p}\right) \sim \tfrac{e^{-\gamma}}{\log x}$ \\ \hline
        Cyclotomic extension & Dirichlet's theorem  & Williams' theorem \\
         $E = \mathbb{Q}(\zeta_b),  F = \mathbb{Q}$ & $\displaystyle\sum_{\substack{p \, \leq \, x \\ p \, \equiv \, a \md b}} 1 \sim \tfrac{1}{\varphi(b)}\tfrac{x}{\log x}$ & $\displaystyle\prod_{\substack{p \, \leq \, x \\ p \, \equiv \, a \md b}}\left(1 - \tfrac{1}{p}\right) \sim \left(\tfrac{e^{-\gamma(a,b)}}{\log x}\right)^{1/\varphi(b)}$ \\ \hline
        Number field & Laundau's theorem  & Rosen's theorem \\
         $E = F \supseteq \QQ$ & $\displaystyle\sum_{\nrm P \, \leq \, x } 1 \sim \tfrac{x}{\log x}$ & $\displaystyle\prod_{\nrm P \, \leq \, x }\left(1 - \tfrac{1}{\nrm P}\right) \sim \tfrac{e^{-\gamma_E}}{\log x}$ \\ \hline
        Galois extension & Chebotarev's theorem & Equation \ref{eq:main}\\
        $E \supseteq F \supseteq \QQ$ & $\displaystyle\sum_{\substack{\nrm P\, \leq \, x \\ \Frob_P = C}} 1 \sim \tfrac{|C|}{|G|}\tfrac{x}{\log x}$ & \cellcolor{cyan!10} $\displaystyle\prod_{\substack{\nrm P \, \leq \, x \\ \Frob_P = C}}\left(1 - \tfrac{1}{\nrm P}\right) \sim \left(\tfrac{e^{-\gamma(E/F,C)}}{\log x}\right)^{|C|/|G|}$\\ \hline
    \end{tabular}
\end{table}    
\end{center}

\subsection{Notation}
We use $s$ to denote a complex variable and write $s \colonequals \sigma + i t$ for its real and imaginary parts.

For an algebraic number field $F/ \mathbb{Q}$, let  $\mathscr{O}_F$ be its ring of integers. Given a non-zero integral ideal $I \trianglelefteq \mathscr{O}_F$, we use $\mathrm{N}_F(I) \colonequals \#(\mathscr{O}_F/I)$ and $\varphi_F(I) \colonequals \#(\mathscr{O}_F/I)^\times$ to denote its absolute norm and totient, respectively. We will take $\Sigma_F$ to be the set of maximal ideals of $\OO_F$.

The Dedekind zeta function of $F$ is denoted by $\zeta_F(s)$, and $\varkappa_F$ will stand for its residue at the pole $s=1$.

Given a subset $S\subseteq \Sigma_F$ and a real number $x \geq 2$, we define $S(x) \colonequals \{P \in S : \nrm_F(P) \leq x\}$. If $S$ has a natural density, it will be denoted by $\delta(S)$.

Throughout the paper, $E/ F$ will be a Galois extension of number fields with Galois group $G \colonequals \Gal(E/F)$, and $C \subseteq G$ will be a fixed conjugacy class. The letters $Q$ and $P$ stand for elements of $\Sigma_E$ and $\Sigma_F$, respectively. Moreover, $Q$ will always be a prime of $E$ above $P$. Their respective residue fields are denoted by $\mF_Q$ and $\mF_P$.
\begin{center}
\begin{tikzcd}
	E \arrow[d, "{G=\text{Gal}(E/F)}"', no head] \arrow[r, "{\supset}", phantom]& Q \arrow[d, no head] & & \mathbb{F}_Q \arrow[d, "{\mathrm{Gal}(\mF_Q/\mF_P)}", no head] \\
	F \arrow[r, "{\supset}", phantom] & P & & \mathbb{F}_P
\end{tikzcd}
\end{center}

Denote by $I_Q \trianglelefteq D_Q \subseteq G$ the inertia and decomposition groups of a prime $Q$ above $P$. Choosing another prime above $P$, the corresponding inertia and decomposition groups are $G$-conjugates of $I_Q$ and $D_Q$. We choose a \cdef{Frobenius element} $\Frob_Q \in D_Q$ whose image in $\Gal(\mF_Q/\mF_P)$ is the cyclic generator. Frobenius elements are only defined modulo the inertia subgroup. Recall that $P\in \Sigma_F$ is unramified in $E$ if and only if $P$ does not divide the discriminant ideal $\Delta \colonequals \Delta_{E/F}$. We will denote the set of unramified primes by $S_{E/F} \subseteq \Sigma_F$. For an unramified prime, $P$, we denote by $\Frob_P$ the \cdef{Frobenius conjugacy class} of Frobenius elements at all primes $Q$ above $P$. 
Given a conjugacy class $C\subseteq G$, let $\cC$ be the set of unramified primes in $\Sigma_F$ with Frobenius conjugacy class equal to $C$.

Let $\rho\colon G \to \GL_n(\CC)$ be a representation of $G$ with underlying vector space $V$. We will use $\chi$ to denote the trace of $\rho$. Given any prime $P\in \Sigma_F$, and $Q\in\Sigma_E$ above $P$, let
\begin{equation}
  \label{eq:L_P}
  L_P(s, \chi, F) \colonequals
  \det\paren{I-\rho(\Frob_Q)|_{V^{I_Q}}(\nrm P)^{-s}}\inv \, , \mfor
  \sigma > 1,
\end{equation}
be the \cdef{Artin Euler factor} at $P$. The \cdef{Artin $L$-function} of $\chi$ is defined for $\sigma > 1$ by the Euler product $L(s, \chi, F) \colonequals \prod_{P\in\Sigma_F}L_P(s, \chi, F)$. We will use the facts that $L(s, \chi, F)$ has a meromorphic extension to the complex numbers, and if $\chi$ is a nontrivial character then $L(1, \chi, F) \neq 0$. When a Galois extension $E/F$ is fixed, we abbreviate $L(s,\chi,F)$ to $L(s, \chi)$. For a comprehensive introduction to the topic of Artin $L$-functions, see \cite{murtymurty2012}. 

\subsection{Main result}
Now that we have the necessary notation in place, we are ready to state our main result.
\begin{thmx}
  \label{thm:main}
  Let $E/ F$ be a Galois extension of number fields, with Galois group $G \colonequals \Gal(E/F)$, and let $C \subset G$ be a conjugacy class. Then,
  \begin{equation}
  \label{eq:3}
    \prod_{P \in \cC(x)} \paren{1-\dfrac{1}{\nrm P}} =
    \paren{\dfrac{e^{-\gamma(E/F,C)}}{\log
    x}}^{|C|/|G|}+O\paren{\dfrac{1}{(\log x)^{\delta(\cC) + 1}}}
\end{equation}
  when $x \to \infty$, and the implied constant depends on the extension E/F and C. Furthermore, the constant $e^{-\gamma(E/F,C)}$ is given by
  \begin{equation}\label{eq:altconstant}
      e^{-\gamma(E/F, C)} = e^{-\gamma_F} \prod_{P \in \Sigma_F} \paren{1-\dfrac{1}{\nrm P}}^{\alpha(E/F, C; P)}
  \end{equation}
  where $\gamma_F \colonequals \gamma + \log \varkappa_F$, and
  \begin{equation}
      \alpha(E/F, C; P) = \begin{cases}
      -1, & P \mid \Delta, \\
      \tfrac{|G|}{|C|}-1, & \Frob_P = C, \\
      -1, & \Frob_P \neq C.
      \end{cases}
  \end{equation}
\end{thmx}

A result similar to Theorem \ref{thm:main}, which we became aware of after finishing our work, appears in Section 8 of an unpublished survey article by Bardestani and Freiberg \cite{bardestanifreiberg}. The approach sketched there follows an adaptation of a proof of Mertens' Theorem due to Hardy, while our method closely follows the strategy of Williams and obtains both an improved statement for the error term and a description of the constants involved.  

\begin{remark}[Error terms]
    The error term $O\left( \frac{1}{(\log x)^{\delta(\cC) + 1}}\right)$ in \eqref{eq:3} agrees with that given by Williams \cite{williams1974}. In the case of a cyclotomic extension $\QQ(\zeta_b)/\QQ$, the error term may be improved by studying zero-free regions of Dirichlet L-functions; see \cite{langzacc}. Assuming the generalized Riemann hypothesis (GRH), one can improve this error term all the way to $O\left( \frac{(\log x)^{1 - \varphi(b)}}{\sqrt{x}}\right)$ \cite[Theorem 4]{langzacc}. Assuming GRH, one also obtains similarly sharp error estimates for \eqref{eq:rosen}, Mertens' theorem over number fields; see \cite[Theorem 7]{lebacque2007}. In order to carry these improvements to the error term in \eqref{eq:3} to the general case, we need faster convergence of $L(1, \chi)$ for irreducible non-trivial $\chi$ than what we use in Theorem \ref{thm:RosenThm2} (see \cite[Theorem 5]{rosen1999}).
\end{remark}

\subsection{Layout} In Section 2, we summarize the work of Williams and Rosen and prove some supporting lemmas. In Section 3 we prove Theorem \ref{thm:main}. In Section 4 we provide some examples.

\section{Background}
\subsection{Williams' argument}
Consider momentarily the case of a cyclotomic extension. Let $b$ be a positive integer, and choose $0 < a < b$ coprime to $b$. In \cite{williams1974}, Williams proved 
\begin{equation}
  \label{eq:williams}
  \prod_{\substack{p \, \leq \, x \\ p \, \equiv \, a \md b}}\paren{1-\dfrac{1}{p}}
  = \paren{\dfrac{e^{-\gamma(a,b)}}{\log x}}^{1/\varphi(b)} +
  O_b\paren{\dfrac{1}{(\log x) ^{1/\varphi(b)+1}}}.
\end{equation}
Furthermore, he was able to give a formula for the constant $\gamma(a,b)$ in terms of
\begin{itemize}
\item the Euler constant $\gamma \colonequals \lim_{s \to 1^+}\left(\zeta(s) - \frac{1}{s-1}\right)$;
\item the ramified primes of the extension $\QQ(\zeta_b)/\QQ$, namely $\prod_{p\mid b}(1-p\inv)\inv = b/\varphi(b)$;
\item the values at $s=1$ of the Dirichlet $L$-functions $L(s,\chi)$, for all non-trivial irreducible characters $\chi$ of the Galois group $\Gal(\QQ(\zeta_b)/\QQ) \cong (\ZZ/b\ZZ)\unit$;
\item the values at $s=1$ of some auxiliary functions $K(s,\chi)$, attached to all non-trivial irreducible characters $\chi$ of the Galois group $\Gal(\QQ(\zeta_b)/\QQ) \cong (\ZZ/b\ZZ)\unit$.
\end{itemize}
Explicitly,
\begin{equation}
  \label{eq:williams_constant}
  e^{-\gamma(a,b)} \colonequals e^{-\gamma} \dfrac{b}{\varphi(b)} \prod_{\chi\neq \chi_0}\paren{\dfrac{K(1,\chi)}{L(1,\chi)}}^{\overline{\chi}(a)}.
\end{equation}

The crux of the proof is to use the orthogonality relations between irreducible characters of finite groups to write

\begin{equation}
  \label{eq:williamsOrthogonalityRelations}
  \prod_{\substack{p \, \leq \, x \\ p \, \equiv \, a \md
    b}}\paren{1-\dfrac{1}{p}}^{\varphi(b)} = \prod_{\chi}
\brc{\prod_{p \, \leq \, x} \paren{1-\dfrac{1}{p}}^{\chi(p)}}^{\overline{\chi}(a)},
\end{equation}
where $\chi$ ranges over all the irreducible characters of $(\ZZ/b\ZZ)\unit$. Afterward, he defines an arithmetic function $k_\chi(n)$ for each $\chi$ that allows one to factor out the Euler factors of the Dirichlet $L$-function as follows

\begin{equation}
  \label{eq:factor}
  \paren{1-\dfrac{1}{p}}^{\chi(p)} =
  \paren{1-\dfrac{\chi(p)}{p}}\paren{1-\dfrac{k_\chi(p)}{p}}\inv.
\end{equation}
Defining
\begin{equation}
  \label{eq:williamsK}
  K(s,\chi) \colonequals \prod_p K_P(s,\chi) 
  \colonequals \prod_p\paren{1-\dfrac{k_\chi(p)}{p^s}}\inv,
  \quad \mfor \sigma > 0,
\end{equation}
the argument then reduces to calculating the asymptotics as $x \to \infty$ of the partial products
\begin{align*}
  \prod_{p \, \leq \, x} L_P(1,\chi)\inv \quad \mand \quad \prod_{p \, \leq \, x} K_P(1,\chi).
\end{align*}
When $\chi = \chi_0$, the calculation follows from Mertens'
theorem. For non-trivial characters, the result follows by standard methods. See \cite{williams1974} for additional details. \\

Back to the general case of an arbitrary Galois extension of number fields $E/ F$, with Galois group $G$, and $C \subseteq G$ a fixed conjugacy class, essentially the same argument works when all the irreducible representations of $G$ are one dimensional (e.g. the case of abelian extensions). However, for higher dimensional representations, we are led to consider a linear approximation of Artin's $L$-function, which we call $M(s,\rho)$ for alphabetical reasons.

\subsection{Rosen's Work}

 Our goal is now to extend the tools used by Williams' to the case of arbitrary Galois extensions. The following theorems of Rosen in \cite{rosen1999} gives estimates of analogues of the partial products of $L_P(1, \chi)$ and $\zeta_F(s)$ as above. 

Rosen's generalization of Mertens' theorem is analogous to the so called Prime Ideal Theorem; Landau's generalization of the Prime Number Theorem to prime ideals in number fields.
\begin{theorem}[Theorem 2 in \cite{rosen1999}]
  \label{thm:RosenThm2}
  Let $F / \QQ$ be an algebraic number field. Then,
  \begin{equation}
  \label{eq:Rosen}
    \prod_{P \in \Sigma_F(x)} \paren{1-\dfrac{1}{\nrm P}} =
    \dfrac{e^{-\gamma_F}}{\log x} + O\paren{\dfrac{1}{\log^2 x}},
\end{equation}
  as $x \to \infty$. Furthermore, $\gamma_F = \gamma + \log \varkappa_F,$ and the implied constant in the error term depends only on the number field $F$.
\end{theorem}

Rosen's proof of Theorem \ref{thm:RosenThm2} extends to a general class of Dirichlet series based on $F$ (Theorem 4 in \cite{rosen1999}). In particular, he proves a Mertens-type theorem for Artin $L$-functions based of $F$, see Theorem 5 in \cite{rosen1999}. We reformulate the original statement in the equivalent case of an irreducible representation.

\begin{theorem}[See Theorem 5 of \cite{rosen1999}]
  \label{thm:MertensArtin}
  Let $E/F$ be a Galois extension of number fields with Galois group $G$. Let $\chi$ be a non-trivial irreducible character of $G$. Then,
  \begin{equation}
  \label{eq:MertensArtin}
    \prod_{P\in\Sigma_F(x)} L_P(1, \chi)\inv = \dfrac{1}{L(1,\chi)} +
    O_{\chi, \, F}\paren{\dfrac{1}{\log x}} .
\end{equation}
\end{theorem}
\begin{proof}
  We are specializing Theorem 5 of \cite{rosen1999} to the case of an irreducible and non trivial character. In Rosen's notation, $\rho = \chi$, $k=0$, and $\alpha = L(1,\chi)$.
\end{proof}

\subsection{The $M$-function}
Let $P \in \Sigma_F$, and let $\rho$ be an Artin representation of $G$. Let
\begin{equation}
  \label{eq:charpoly}
  f_{\chi,P}(T) \colonequals \det\paren{I-\rho(\Frob_Q)|_{V^{I_Q}}T} \in \CC[T]
\end{equation}
be the characteristic polynomial of $\rho$ corresponding to $P$ via any Frobenius element. Denote the \cdef{trace of Frobenius at} $P$ by 
\begin{equation}
  \label{eq:traceOfFrobenius}
  \chi(P) \colonequals \Tr\rho(\Frob_Q)|_{V^{I_Q}}
\end{equation}
for any prime $Q\in \Sigma_E$ above $P\in\Sigma_F$. Isolating the linear term, we have
\begin{equation}
  \label{eq:linearterm}
  f_{\chi,P}(T) = 1 - \chi(P)T + g_{\chi,P}(T)T^2,
\end{equation}
where $g_{\chi, P}(T) \in \CC[T]$. Factoring out the linear term, we may write
\begin{equation}
  \label{eq:factors}
  f_{\chi,P}(T) = \paren{1 - \chi(P)T}\paren{1 +
  \dfrac{g_{\chi,P}(T)T^2}{1 - \chi(P)T}} \in \CC(T) .
\end{equation}
Taking the change of variables $T = (\nrm P)^{-s}$, we obtain
\begin{equation}
  \label{eq:L_Pfactors}
  L_P(s,\chi) = \paren{1-\dfrac{\chi(P)}{(\nrm P)^s}}\inv\paren{1 +
  \dfrac{g_{\chi,P}( (\nrm P)^{-s})}{(\nrm P)^s((\nrm P)^s - \chi(P))}}\inv, \, \sigma > 1.
\end{equation}
Since $-\log_{NP}(|T|)=\sigma$, we have $0 < |T| \leq \frac{1}{NP}\leq\frac{1}{2}$ when $\sigma \geq 1$. Define
 \begin{equation}\label{eq:chibound}
     \xi_\chi \colonequals \sup_{P \in \Sigma_F} \left( \sup_{T \in [0,1/2]} |g_{\chi,P}(T)| \right)
 \end{equation}
to give an upper bound $g_{\chi,P}(T) \leq \xi_\chi$. This is well defined since $g_{\chi,P}(T)$ depends only on the class of $\text{Frob}_Q$ and the set of the  $g_{\chi, P}$  is finite. 

\begin{defn}[$M$-function]
  Given $P\in \Sigma_F$ and $\rho$ and Artin representation of $G$ with character $\chi$, define the $M$-\cdef{Euler factor at} $P$ by
  \begin{equation}
  \label{eq:M_P}
    M_P(s,\chi) \colonequals \paren{1-\dfrac{\chi(P)}{(\nrm
    P)^s}}\inv, \quad \mfor \sigma > 1.
\end{equation}
  We define the \cdef{$M$-function} as the Euler product $M(s,\chi) \colonequals \prod_{P\in \Sigma_F} M_P(s,\chi)$.
\end{defn}
Note that $M(s,\chi)$ defines an holomorphic function in the half plane $\sigma >1$. When $\rho$ is one dimensional, the polynomial $R_{\chi, P}(T)$ is zero, and in particular $L(s,\chi) = M(s,\chi)$. For higher dimensional representations, this is certainly not the case. We think of $M(s,\chi)$ as a linear approximation of $L(s,\chi)$, at least on the level of local factors.

 As a preliminary step in the proof of Theorem \ref{thm:main}, we prove a Mertens-type theorem for $M(s,\chi)$. Though it would be interesting to further explore the analytic properties of $M(s,\chi)$, we restrict ourselves to applications of $M(s,\chi)$ to the proof of the main theorem.

When the representation $\rho$ is one dimensional, $\chi(P)$ is always a root of unity. For higher dimensional representations, $\chi(P)$ is a sum of roots of unity and $|\chi(P)|\leq \chi(1)$. In particular, it may be the case that $\chi(P) = \nrm P$. To deal with these technicalities, we restrict to a cofinite subset $\cS$ of $\Sigma_F$ over which this inconvenience disappears. Define
\begin{equation}
  \label{eq:S}
  \cS \colonequals \{P \in S_{E/F} : |\chi(P)| < \nrm P, \, \text{for
  every irreducible character } \chi \text{ of } G\}.
\end{equation}
\begin{lemma}
  \label{lemma:M}
  Let $E/F$ be a Galois extension of number fields with Galois group $G$. Let $\chi$ be a non-trivial irreducible character of $G$ and let $\cS$ be as in \eqref{eq:S}. Then,
  \begin{equation}
  \label{eq:6}
    \prod_{P\in \cS(x)} M_P(1,\chi)\inv = \dfrac{R_{\cS,\chi}M_{\cS,\chi}}{L(1,\chi)} +
    O\paren{\dfrac{1}{\log x}}, 
\end{equation}
  when $x\to \infty$ and the implied constant in the error term depends on the extension $E/F$. Furthermore, the constants $R_{\cS, \chi}$ and $M_{\cS,\chi}$ are given by
  \begin{equation}
  \label{eq:1}
      R_{\cS, \chi} = \prod_{P \in \cS}R_P(1, \chi)^{-1} \text{ and }
    M_{\cS,\chi} = \prod_{P\in \Sigma_F - \cS}L_P(1,\chi).
\end{equation}
\end{lemma}
\begin{proof}
Factoring the linear term of the characteristic polynomial of Frobenius
\begin{equation}
    f_{\chi, P}(T) = (1-\chi(P)T)\left(1 +\frac{g_{\chi,P}(T)T^2}{1-\chi(P)T}\right),
\end{equation}
and define,
\begin{equation}\label{eq:R_P}
    R_P(1, \chi) \colonequals \left(1+\frac{g_{\chi,P}(\nrm P^{-1})NP^{-2}}{1-\chi(P)NP^{-1}}\right) = \left(1+\frac{g_{\chi,P}(\nrm P^{-1})}{NP(NP-\chi(P))}\right).
\end{equation}
Combining \eqref{eq:R_P} with \eqref{eq:L_Pfactors} and \eqref{eq:M_P} gives us that 
\begin{equation}
    L_P(1, \chi) = M_P(1, \chi)R_P(1,\chi)^{-1}.
\end{equation}
Taking a product over $P \in \cS(x)$ of the above expression, we get
\begin{equation} \label{eq:MpAsLpEp}
    \prod_{P \in \cS(x)} M_P(1, \chi)^{-1} = \prod_{P \in \cS(x)}L_P(1,\chi)^{-1} \prod_{P \in \cS(x)}R_P(1,\chi)^{-1}.
\end{equation}
We can use Theorem \ref{thm:MertensArtin} to understand the product of $L_P(1,\chi)$. Doing so, one sees 
\begin{align}
\notag    \prod_{P \in \cS(x)} L_P(1,\chi)^{-1} &=     \prod_{P \in \Sigma_F(x)} L_P(1,\chi)^{-1} \prod_{P\in \Sigma_F(x)-\cS(x)} L_P(1,\chi) \\
    &= \left(\prod_{P\in \Sigma_F(x)-\cS(x)} L_P(1,\chi)\right)\left(\frac{1}{L(1,\chi)} + O_{\chi,F}\left(\frac{1}{\log x}\right)\right) \label{eq:prodLpS}
\end{align}

What remains is to understand the product over $R_P(1,\chi)^{-1}$, i.e.,
\begin{equation}\label{eq:Eprod}
\prod_{P \in \cS(x)} R_P(1,\chi)^{-1} = \prod_{P \in \cS(x)}\left(1+\frac{g_{\chi,P}(\nrm P^{-1})}{NP(NP-\chi(P))}\right)^{-1},
\end{equation}
and to show the product of $R_P(1,\chi)^{-1}$ over all $P \in \cS$ converges, say to 
\begin{equation*}
    R_{\cS, \chi} \colonequals \prod_{P \in \cS}R_P(1, \chi)^{-1}.
\end{equation*}

We have, for large enough $x$,
\begin{equation}
    \prod_{P \in \cS(x)} R_P(1,\chi)^{-1} = R_{\cS, \chi}\prod_{\substack{ \nrm P > x}}R_P(1,\chi).
\end{equation}
To understand the right-most term above, take logs and expand via Taylor series as follows, 
\begin{align}
\notag    \left| \log\left(\prod_{\substack{P \in \cS \\ \nrm P > x}} R_P(1,\chi)\right) \right| &= \left|\sum_{\substack{P \in \cS \\ \nrm P > x}} \log\left(1+\frac{g_{\chi,P}((\nrm P)^{-1})}{\nrm P (\nrm P - \chi(P))}\right)\right| \\
\notag    &\leq \sum_{\nrm P > x} \sum_{j=1}^\infty \frac{1}{j}\left|\frac{\xi_{\chi}}{\nrm P (\nrm P - \chi(P))}\right|^j \\
\notag     &\leq \sum_{j=1}^\infty \frac{1}{j} \left(\sum_{\nrm P > x} \frac{\xi_{\chi}}{\nrm P (\nrm P - \chi(P))}\right)^j \\
     &\leq \sum_{j=1}^\infty \frac{1}{j} \left(\sum_{\nrm P > x} O\left(\frac{1}{(\nrm P )^2}\right)\right)^j =\sum_{j=1}^\infty\frac{1}{j}\left(O\left(\frac{1}{x}\right)\right)^j =O\left(\frac{1}{x}\right). \label{eq:log1term}
\end{align}
Where the first equality in \eqref{eq:log1term} follows from the prime ideal theorem and partial summation. From \eqref{eq:log1term} and the Taylor series of the exponential, observe $\exp\left(O\left(\frac{1}{x}\right)\right)= 1 + O\left(\frac{1}{x}\right)$. This is sufficient to establish  
\begin{equation}
    \prod_{\nrm P > x}R_P(1,\chi) = 1 + O\left(\frac{1}{x}\right),
\end{equation}
and further
\begin{equation}\label{eq:prodEpS}
    \prod_{P \in \cS(x)}R_P(1,\chi)^{-1} = R_{\cS, \chi} + O\left(\frac{1}{x}\right).
\end{equation}

Finally, starting from \eqref{eq:MpAsLpEp} and substituting in both \eqref{eq:prodLpS} and \eqref{eq:prodEpS} appropriately suffices to prove the lemma.
\end{proof}

\subsection{The $K$-function} 

In this section, we investigate the analog of Williams' $K$-function, defined in the case of cyclotomic extensions by \eqref{eq:factor} and \eqref{eq:williamsK}.

\begin{defn}[$K$-function]
  Given $P\in \Sigma_F$ and $\rho$ an Artin representation of $G$ with character $\chi$, define
  \begin{equation}
  \label{eq:k}
    k_\chi(P) \colonequals \nrm P \brc{1 -
  \paren{1-\dfrac{\chi(P)}{\nrm P}}\paren{1-\dfrac{1}{\nrm P}}^{-\chi(P)}}.
\end{equation}
  The $K$-\cdef{Euler factor at} $P$ is defined by
  \begin{equation}
  \label{eq:K_P}
    K_P(s,\chi) \colonequals \paren{1-\dfrac{k_\chi(P)}{(\nrm
    P)^s}}\inv,
\end{equation}
  and we define the $K$-\cdef{function} as the Euler product $K(s,\chi) \colonequals \prod_{P\in \Sigma_F}K_P(s,\chi)$.
\end{defn}
Note that $|\chi(P)| < \nrm P$, so $K_P(1,\chi)$ is well defined and non-zero for every prime $P\in \Sigma_F$. 

To prove Theorem \ref{thm:main} it is enough to restrict the Euler product in the definition of $K$ to the primes in $\cS$.

We first want to know that the truncated product $\prod_{P \in \cS(x)}K_P(1,\chi)$ converges to $\prod_{P \in \cS}K_P(1,\chi)$ quickly, in a precise sense. This is the statement of Lemma \ref{lemma:K}. To prove this, we will need some intermediate lemmas. The following lemma is implicit in \cite{williams1974} and will be critical to the analysis of $k_\chi(p)$.

\begin{lemma}\label{lemma:nonsense}
  Let $a,b$ be complex numbers such that $|a/b| < 1$ and $b \geq 2$. Then,
  \begin{equation}
  \label{eq:technicalLemma}
    b\brc{1-\paren{1-\dfrac{a}{b}}\paren{1-\dfrac{1}{b}}^{-a}} =
    \dfrac{a(a-1)}{b}\brc{\dfrac{1}{2} + \sum_{n=1}^\infty \dfrac{(a+1)\cdots(a+n)}{b^n(n+1)!}\dfrac{n+1}{n+2}}.  
\end{equation}
\end{lemma}

While we omit the details of the proof of Lemma \ref{lemma:nonsense}, we comment that it follows from writing the left hand side as a doubly infinite series. We can then rearrange this series into a power series in $\frac{1}{b}$, then use induction to show that the coefficients take the desired form. The key application of Lemma \ref{lemma:nonsense} is the following estimate.

\begin{lemma}\label{lem:bound_kchi} Let $\chi$ be a $d$-dimensional irreducible character of $G$. Then for all $P \in \cS$,
    \begin{equation}\label{eq:bound_kchi}
        \left| k_\chi(P) \right| \leq \frac{d(d+1)}{2} \frac{1}{\nrm P} + \dfrac{C_d}{(\nrm P)^2}
    \end{equation}
    for some constant $C_d>0$ depending only on $d$.
\end{lemma}

\begin{proof}
    Fix $P\in \cS$. The conditions of Lemma \ref{lemma:nonsense} are satisfied for $a = \chi(P)$ and $b=\nrm P$. Noting that $|\chi(P)| \leq d$, we have
    \begin{align*}
        |k_\chi(P)| = & \, \left| \dfrac{\chi(P)(\chi(P)-1)}{\nrm P} \brc{\dfrac{1}{2} + \sum_{n=1}^\infty \dfrac{(\chi(P)+1)\cdots(\chi(P)+n)}{(\nrm P)^n(n+1)!}\dfrac{n+1}{n+2}}\right| \\
        \leq & \, \dfrac{d(d+1)}{2}\dfrac{1}{\nrm P} + \dfrac{d(d+1)}{(\nrm P)^2} \sum_{n=0}^\infty \dfrac{(n-1+d)!}{d! \, n!}\dfrac{1}{(\nrm P)^n}\, \\
        \leq & \, \dfrac{d(d+1)}{2}\dfrac{1}{\nrm P} + \paren{\dfrac{d(d+1)}{d!} \sum_{n=0}^\infty \dfrac{(n-1+d)!}{n!}\dfrac{1}{2^n}} \dfrac{1}{(\nrm P)^2} .
    \end{align*}
    The constant $C_d$ is given by the expression inside the big parenthesis in the last inequality,
    \begin{equation}\label{eq:Cd}
        C_d = \dfrac{d(d+1)}{d!} \sum_{n=0}^\infty \dfrac{(n-1+d)!}{n!}\dfrac{1}{2^n}.
    \end{equation}
    To determine the convergence of the series, it is enough to notice that $(n-1+d)!/n!$ is a polynomial of degree $d-1$ in $n$.
\end{proof}

Combining the estimate of Lemma \ref{lem:bound_kchi} with the prime ideal theorem, we have the following estimate on the tail of the infinite sum of $|k_\chi(P)|/\nrm P$ over primes $P \in \cS$.

\begin{lemma}\label{lem:k_partial_sum} Let $x > 0$. Then
    \[\sum_{\substack{P \in \cS \\ \nrm P>x}} \frac{\left|k_\chi(P)\right|}{\nrm P} = O\left(\frac{1}{x} \right)\]
    where the implied constant depends on the extension $E/F$.
\end{lemma}

\begin{proof}
    Notice that when $x$ is sufficiently large, all primes $P$ with norm $\nrm P > x$ are contained in $\cS$. This allows us to drop the requirement in the summation that $P \in \cS$.
    
    By \eqref{eq:bound_kchi} we have
    \[\sum_{\nrm P>x} \frac{\left|k_\chi(P)\right|}{\nrm P} \leq \frac{d^2+d}{2} \sum_{\nrm P > x} \frac{1}{\nrm P^2} + C_d \sum_{\nrm P > x} \frac{1}{\nrm P^3}\]
    for the constant $C_d$ depending on the dimension of the representation associated to $\chi$ given above in \eqref{eq:Cd}.    Of course, $1/\nrm P^3 < 1/\nrm P^2$, so it suffices to show that $\sum_{\nrm P > x} \frac{1}{\nrm P^2} = O(1/x)$. This follows from the same argument as in \eqref{eq:log1term}.

\end{proof}

\begin{lemma}
  \label{lemma:K}
  Let $E/F$ be a Galois extension of number fields with Galois group $G$. Let $\chi$ be a non-trivial irreducible character of $G$, and let $\cS$ be as defined in \eqref{eq:S}. Then
  \begin{equation}
  \label{eq:assymptoticK}
    \prod_{P\in \cS(x)} K_P(1,\chi) = K_{\cS, \chi} + O\paren{\dfrac{1}{x}},
\end{equation}
  when $x \to \infty$ and the implied constant depends on the extension $E/F$. Furthermore, the constant $K_{\cS, \chi}$ is given by
  \begin{equation}
  \label{eq:K_S}
    K_{\cS, \chi} \colonequals \prod_{P\in \cS} K_P(1,\chi).
\end{equation}
\end{lemma}
\begin{proof}
    The set $\cS$ as defined in \eqref{eq:S} provides that $|\chi(P)| < \nrm P$, so for all $P \in \cS$, we can see by \eqref{eq:k} that $k_\chi(P) \neq \nrm P$, and hence by \eqref{eq:K_P} the local factor $K_P(1,\chi)$ is both well defined and nonzero. Our goal is thus to show that the infinite product $K_{\cS, \chi}$ converges and that the truncation $\prod_{P \in \cS(x)} K_P(1, \chi)$ converges to it with the error term $O_F\left(\frac{1}{x} \right)$.
    
    Taking logarithms, the limit of 
    \[\log\left(\prod_{P \in \cS(x)} K_P(1,\chi) \right) = -\sum_{P \in \cS(x)} \log \left(1 - \frac{k_\chi(P)}{\nrm P} \right) \]
    converges as $x \to \infty$. To see this, and obtain the desired asymptotic, it suffices to estimate the tail
    \[\left| \sum_{\substack{P \in \cS\\ \nrm P > x}} \log \left(1 - \frac{k_\chi(P)}{\nrm P} \right)\right|.\]
    For $x$ sufficiently large, all primes of sufficiently large norm are in $\cS$, so it suffices to estimate this tail for all $\nrm P > x$. Taking absolute values and using the Taylor series expansion, which is valid since $P \in \cS$, we have
    \begin{align}
       \nonumber \left| \sum_{\nrm P > x} \log \left(1 - \frac{k_\chi(P)}{\nrm P} \right) \right| & \leq \sum_{\nrm P > x} \left| \sum_{j=1}^\infty \frac{1}{j} \left(\frac{k_\chi(P)}{\nrm P}\right)^j\right|  \\
       \notag & \leq  \sum_{\nrm P > x} \sum_{j=1}^\infty \frac{1}{j} \left(\frac{ \left|k_\chi(P)\right|}{\nrm P}\right)^j  \\
       & \leq \sum_{j=1}^\infty \frac{1}{j} \sum_{\nrm P > x} \left(\frac{ \left|k_\chi(P)\right|}{\nrm P}\right)^j  \nonumber \\ 
       \notag & \leq \sum_{j=1}^\infty \frac{1}{j} \left(\sum_{\nrm P > x} \frac{ \left|k_\chi(P)\right|}{\nrm P}\right)^j  = \sum_{j=1}^\infty \frac{1}{j} O_F\left(\frac{1}{x} \right)^j .
    \end{align}
    The last equality follows from Lemma \ref{lem:k_partial_sum}. As in \eqref{eq:log1term}, this establishes that $\prod_{\nrm P > x} K_P(1, \chi) = 1 + O\left(\frac{1}{x}\right)$, completing the proof of \eqref{eq:assymptoticK}.
    
    
\end{proof}

\section{Proof of Theorem \ref{thm:main}}

In this section we will first prove the content of Theorem \ref{thm:main} and then show an alternative determination of the constant following a method shown in \cite[Section 6]{langzacc}. 

\subsection{Proof of the main theorem}
\label{subsection:proofmaintheorem}

The starting point of our proof is the same of Williams, namely, the orthogonality relations for irreducible characters of finite groups. Given a fixed conjugacy class $C$ of $G$, and an unramified prime $P\in S_{E/F}$, we have
\begin{equation}
  \label{eq:OR}
  \sum_{\chi} \chi(P)\overline{\chi}(C) = 
  \begin{cases}
  \frac{|G|}{|C|}, & \mif C = \Frob_P, \\
    0, & \mif C \neq \Frob_P.
\end{cases}
\end{equation}
This leads to the natural generalization of Equation \eqref{eq:williamsOrthogonalityRelations}.
\begin{equation}
  \label{eq:OrthogonalityRelations}
  \prod_{P\in S_C(x)}\paren{1-\dfrac{1}{\nrm P}}^{|G|/|C|} = \prod_{\chi}
  \brc{\prod_{P\in S_{E/F}(x)} \paren{1-\dfrac{1}{\nrm P}}^{\chi(P)}}^{\overline{\chi}(C)}.
\end{equation}
When $\chi = \chi_0$ is the trivial character, Rosen's theorem (Theorem \ref{thm:RosenThm2}) yields
\begin{equation}
  \label{eq:RosenContribution}
  \prod_{P\in S_{E/F}(x)} \paren{1-\dfrac{1}{\nrm P}} = \dfrac{\nrm
  \Delta}{\varphi(\Delta)} \dfrac{e^{-\gamma_F}}{\log x} +
O_F\paren{\dfrac{1}{\log^2 x}} .
\end{equation}
When $\chi\neq \chi_0$, we first split the product as follows
\begin{equation}
  \label{eq:goodbad}
  \prod_{P\in S_{E/F}(x)}\paren{1-\dfrac{1}{\nrm P}}^{\chi(P)} =
  \prod_{P\in S_{E/F}(x) - \cS(x)} \paren{1-\dfrac{1}{\nrm P}}^{\chi(P)} \prod_{P\in \cS(x)} \paren{1-\dfrac{1}{\nrm P}}^{\chi(P)}.
\end{equation}
Call $B_{\cS, \chi}$ the constant given by the product over the primes $P \in S_{E/F} - \cS$ in the right hand side of \eqref{eq:goodbad}. For every $P\in \cS$ we are able to factor out $M_P(1,\chi)$ from the expression, obtaining
\begin{align}
  \prod_{P\in \cS(x)} \paren{1-\dfrac{1}{\nrm P}}^{\chi(P)} = &
   \prod_{P\in \cS(x)} M_P(1,\chi)\inv \prod_{P\in \cS(x)} K_P(1,\chi)
   \label{eq:MK} \\[1em]
  = & \brc{ \dfrac{R_{\cS,\chi}M_{\cS,\chi}}{L(1,\chi)} +
   O\paren{\dfrac{1}{\log x}}} \brc{K_{\cS,\chi} +
   O\paren{\dfrac{1}{x}}} \label{eq:lemmas} \\[1em]
  = &
   \dfrac{R_{\cS,\chi}M_{\cS,\chi}K_{\cS,\chi}}{L(1,\chi)} +
   O\paren{\dfrac{1}{\log x}} .
\end{align}
The equality in \eqref{eq:lemmas} follows from applying Lemma \ref{lemma:M} and Lemma \ref{lemma:K}. Again, the constants only depend on the extension $E/F$.  Assembling the pieces together, we get the desired result. 

Finally, the constant $-\gamma(E/F, C)$ is defined by the equality
\begin{equation}
  \label{eq:constant}
  e^{-\gamma(E/F,
  C)} = \dfrac{\nrm \Delta}{\varphi(\Delta)} \prod_{\chi \neq
    \chi_0}\paren{\dfrac{B_{\cS,\chi}R_{\cS,\chi}M_{\cS,\chi}K_{\cS,\chi}}{L(1,\chi)}}^{\overline{\chi}(C)}
e^{-\gamma_F} .
\end{equation}
Note that the constants $B_{\cS, \chi}$ and $M_{\cS, \chi}$ are easily computed finite products. To obtain a numerical value for $e^{-\gamma(E/F, C)}$ for a given, $E/F$ and $C \subset G$, one would need to compute these along with the infinite products $R_{\cS, \chi}$, $K_{\cS, \chi}$, and the $L$-function $L(1, \chi)$ for each nontrivial character $\chi$ of $G$.

\subsection{An alternative determination of the constant}

Languasco and Zaccagnini observed in \cite{langzacc} that the orthogonality relations of finite group characters can also be used to provide a cleaner formula for the constant $e^{-\gamma(a,b)}$ appearing in Williams' theorem. Their method extends to this setting as well, and we record it here for completeness. 

First, note 
\begin{equation*}
    \lim_{x \to \infty} \prod_{P \in  S_{E/F}(x)}\left(1- \frac{1}{\nrm P}\right)^{\chi(P)} = \frac{K(1,\chi)}{L(1,\chi)}.
\end{equation*}
Thus, from \eqref{eq:constant}, 
\begin{align}
 \notag   e^{-\gamma(E/F, C)} &= e^{-\gamma_F}\frac{\nrm \Delta}{\varphi(\Delta)} \lim_{x \to \infty} \prod_{\chi \neq \chi_0}\prod_{P \in S_{E/F}(x)} \left(1- \frac{1}{\nrm P}\right)^{\chi(P)\overline{\chi(C)}} \\
 \notag   &= e^{-\gamma_F}\frac{\nrm \Delta}{\varphi(\Delta)} \lim_{x \to \infty} \prod_{P \in S_{E/F}(x)} \left(1- \frac{1}{\nrm P}\right)^{\sum_{\chi \neq \chi_0}\chi(P)\overline{\chi(C)}} \\
 \label{eq:nrn/phi}  &=  e^{-\gamma_F} \lim_{x \to \infty} \prod_{P \in \Sigma_{E/F}(x)} \left(1- \frac{1}{\nrm P}\right)^{\alpha(E/F, C; P)}\\
\notag    &= e^{-\gamma_F} \prod_{P \in \Sigma_F} \paren{1-\dfrac{1}{\nrm P}}^{\alpha(E/F, C; P)}
\end{align}
where, using character orthogonality \eqref{eq:OR}, 
$$
\alpha(E/F, C; P) = \begin{cases}
      -1, & P \mid \Delta, \\
      \tfrac{|G|}{|C|}-1, & \Frob_P =C, \\
      -1, & \Frob_P \neq C.
      \end{cases}
$$
and \eqref{eq:nrn/phi} follows from the product formula of the Euler totient function
$$\frac{\varphi(\Delta)}{\nrm \Delta}= \prod_{ P \mid \Delta}\left(1-\frac{1}{\nrm P}\right).$$
This calculation is sufficient to prove \eqref{eq:altconstant} of Theorem \ref{thm:main}.
\section{Examples}

\subsection{Quadratic extensions} Set $F = \QQ$ and let $E = \QQ(\sqrt{D})$, with $D$ square-free, be a quadratic extension of $\QQ$. 

\begin{coro}\label{cor:quadratic}
Let $E/\QQ$ be a quadratic extension of discriminant $\Delta$. Then
\[\prod_{\substack{ \left(\frac{D}{p}\right) = \pm 1 \\ p \, \leq \, x}} \left(1 - \frac{1}{p}\right)  = \left( \frac{\Delta}{\varphi(\Delta)} \frac{e^{-\gamma}}{\log x}  \left[ \frac{\prod_{\left(\frac{D}{p}\right) = 1} (1 - \frac{1}{p})}{\prod_{\left(\frac{D}{p}\right) = -1} (1 - \frac{1}{p})} \right]^{\pm 1} \right)^{1/2} + O\left(\frac{1}{(\log x)^{3/2}}\right),\]
where $\gamma = \gamma_\QQ$ is the usual Euler constant.
\end{coro}

In this case, $G = \Gal(E/\QQ) \cong \{\pm 1\}$, so there is one nontrivial conjugacy class $\{-1\} \subseteq G$, consisting of the inert primes in $\cO_E$, while the trivial class corresponds to the split primes. Our two characters are the trivial character, $\chi_0$, and the nontrivial character 
\[\chi_1(p) = \begin{cases}
    \,\,\,\, 1,& \mif \, p \text{ is split},\\
    -1,&  \mif \, p \text{ is inert}.
\end{cases}\]
This is precisely the quadratic residue symbol, $\chi_1(p) = \left(\frac{D}{p}\right)$, which in our notation also coincides with $\Frob_p$. 

Following the algorithm implicit in the proof of the main theorem (Subsection \ref{subsection:proofmaintheorem}), we have $\left|\chi_i(p)\right| = 1 < p$ for both $i=0,1$ and all primes $p$, so $\cS = S_{E/\QQ}$ is precisely the set of unramified primes.

First consider the case where $C = \{1\}$. By \eqref{eq:OrthogonalityRelations} and \eqref{eq:RosenContribution} we have
\[\prod_{\substack{ \left(\frac{D}{p}\right) = 1 \\ p \, \leq \, x}} \left( 1 - \frac{1}{p}\right)^{2} = \frac{\Delta}{\varphi(\Delta)} \frac{e^{-\gamma}}{\log x} \frac{\prod_{\left(\frac{D}{p}\right) = 1} (1 - \frac{1}{p})}{\prod_{\left(\frac{D}{p}\right) = -1} (1 - \frac{1}{p})}  + O\left(\frac{1}{\log^{3} x}\right),\]
where $\gamma$ is the usual Euler constant. Taking square roots, we have
\[\prod_{\substack{ \left(\frac{D}{p}\right) = 1 \\ p \, \leq \, x}} \left( 1 - \frac{1}{p}\right) = \left(\frac{\Delta}{\varphi(\Delta)} \frac{e^{-\gamma}}{\log x}  \frac{\prod_{\left(\frac{D}{p}\right) = 1} (1 - \frac{1}{p})}{\prod_{\left(\frac{D}{p}\right) = -1} (1 - \frac{1}{p})} \right)^{1/2}  + O\left(\frac{1}{\log^{3/2} x}\right).\]
If we took $C=\{-1\}$, then we find
\[\prod_{\substack{ \left(\frac{D}{p}\right) = -1 \\ p \, \leq \, x}} \left( 1 - \frac{1}{p}\right) = \left(\frac{\Delta}{\varphi(\Delta)} \frac{e^{-\gamma}}{\log x}  \frac{\prod_{\left(\frac{D}{p}\right) = -1} (1 - \frac{1}{p})}{\prod_{\left(\frac{D}{p}\right) = 1} (1 - \frac{1}{p})}  \right)^{1/2} + O\left(\frac{1}{\log^{3/2} x}\right).\]

Using Theorem \ref{thm:main} \eqref{eq:altconstant} we obtain the exact same formula for the constant. 

\subsection{Primes represented by quadratic forms} Let
\begin{equation*}
    Q(x,y) = ax^2 + bxy + cy^2 \in \ZZ[x,y],
\end{equation*}
be a binary integral quadratic form. Assume that $Q$ is primitive, irreducible, and positive definite. That is, $a$ and $c$ are positive integers with $\gcd(a,b,c)=1$, $D = b^2 -4ac$ is not a square, and $D<0$. An integer $n$ is said to be represented by $Q$ if there exist integers $x$ and $y$ such that $Q(x,y) = n$. 

Denote by $\cQ$ the set of rational primes represented by $Q$.

\begin{coro}
    Let $Q$ be a primitive, irreducible, positive definite, and integral binary quadratic form with discriminant $D$, and let $E$ be the ring class field of the order of $D$. Then,
    \begin{equation}
        \prod_{p\in \cQ(x)} \paren{1-\dfrac{1}{p}} = 
        \paren{ \dfrac{e^{-\gamma(E/\QQ, \, \cC)}}{\log x}}^{\frac{|C|}{2h(D)}}\prod_{\substack{p \, \mid \, \Delta_E \\ p \, \in \, \cQ}}\paren{1-\dfrac{1}{p}} + O\paren{\dfrac{1}{(\log x)^{1+\frac{|C|}{2h(D)}}}},
    \end{equation}
    where $C \subset \Gal(E/\QQ)$ is the conjugacy class corresponding to $Q$ via class field theory, and $h(D)$ is the class number of $\QQ(\sqrt{D})$.
\end{coro}

\begin{proof}
  By class field theory and the theory of quadratic forms, see for example \cite[Chapter 9]{cox}, the class $[Q]$ corresponds to an element $\sigma_0 \in \Gal(E/\QQ(\sqrt{D})) \subseteq \Gal(E/\QQ)$. Therefore, the class $C$ is the $\Gal(E/\QQ)$-conjugacy class of $\sigma_0$. The result follows by noting that $\cQ - \cC$ is the finite set of primes ramified in $L$ that are represented by $Q$. In particular 
  \begin{equation*}
      \delta(\cC) = \delta(\cQ) = \begin{cases}
      \frac{1}{2h(D)}, & \mif Q \text{ is equivalent to its opposite.}\\
      \frac{1}{h(D)}, & \text{otherwise.}
      \end{cases}
  \end{equation*}
The relation between $\cC$ and $\cQ$ is made explicit in the proof of \cite[Theorem 9.12]{cox} using the ring class field as described in \cite[Section 9.A]{cox}.
\end{proof}

\subsection{General abelian extensions} In the special case the Galois group $G$ is abelian, all irreducible representations are one-dimensional. In particular, the trace of Frobenius is a root of unity, and as such it has absolute value strictly smaller that the norm of every prime. In our notation, this means $\cS=S_{E/F}$. Moreover, the Artin $L$-function coincides with the $M$-function, and we have the following corollary.

\begin{coro}
  \label{cor:abelian}
  Let $E/ F$ be an abelian Galois extension of number fields, with Galois group $G$, and let $g\in G$ be any element. Then,
    \begin{equation}
        \label{eq:abelian}
        \prod_{\substack{P\in S_{E/F}(x)\\ \Frob_P = g}} \paren{1-\dfrac{1}{\nrm P}} =
        \paren{\dfrac{e^{-\gamma(E/F,g)}}{\log
        x}}^{1/[E:F]}+O\paren{\dfrac{1}{(\log x)^{1+1/[E:F]}}}
    \end{equation}
  when $x \to \infty$ and the implied constant depends on the extension $E/F$. Furthermore, the constant $\gamma(E/F,g)$ is given by
  \begin{align*}
      e^{-\gamma(E/F,g)} &= e^{-\gamma_F}\dfrac{\nrm(\Delta)}{\varphi(\Delta)}\prod_{\chi\neq\chi_0}\paren{\prod_{P\nmid \Delta}\dfrac{K_P(1,\chi)}{L_P(1,\chi)}}^{\overline{\chi}(g)} \\
      &= e^{-\gamma_F}\dfrac{\nrm(\Delta)}{\varphi(\Delta)}\prod_{\substack{P \\ \mathrm{Frob}_P=g}}\left(1-\frac{1}{\nrm P}\right)^{[E:F]-1}\prod_{\substack{P \\ \mathrm{Frob}_P\neq g}}\left(1-\frac{1}{\nrm P}\right)^{-1}.
  \end{align*}
\end{coro}

\subsection{Sextic $S_3$-extensions} Finally, we consider the case when $E/\QQ$ sextic $S_3$-extension. We denote the three conjugacy classes of $G$ by the identity class $C_1$, the class of transpositions $C_2$, and the class of 3-cycles $C_3$. The three irreducible characters $\chi_0, \chi_1, \chi_2$ are given by the character table in Figure \ref{fig:char_s3}.

\begin{figure}[h]
    \centering
    \begin{tabular}{|c|c|c|c|}
       \hline & $C_1$ & $C_2$ & $C_3$  \\ \hline
        $\chi_0$ & 1 & 1 & 1 \\ \hline
        $\chi_1$ & 1 & -1& 1 \\ \hline
        $\chi_2$ & 2 & 0 & -1 \\\hline
    \end{tabular}
    \caption{The character table for $S_3$}
    \label{fig:char_s3}
\end{figure}

It is clear from the table that for all \textit{odd} primes $p$, we have $|\chi(p)| < p$, so all odd unramified primes are contained in $\cS$. For the even prime, $2 \notin \cS$ if (i) it is ramified or (ii) if it is unramified and $\chi(2) = 2$ for some $\chi$. From Figure \ref{fig:char_s3}, (ii) can only occur for $\chi_2$ in the case where $\Frob_2$ is the identity class, i.e., precisely when 2 is totally split in $E$. This condition does occur, for example it happens with $p=2$ in the case where $E$ is the splitting field of $x^{6} - 2 x^{5} - 14 x^{3} + 123 x^{2} - 208 x + 164$ over $\QQ$ \cite[Number field \href{https://www.lmfdb.org/NumberField/6.0.80062991.1}{6.0.80062991.1}]{lmfdb6.0.80062991.1}.

This allows us to compute $B_{\cS, \chi}$:
\begin{equation} \label{eq:B_S3}
    B_{\cS, \chi}  = \prod_{p \in S_{E/\QQ} - \cS} \paren{1-\dfrac{1}{\nrm P}}^{\chi(P)} = \begin{cases}
    \frac{1}{2^{\chi(2)}}, & \text{if 2 is unramified and } \Frob_2 = C_1,\\
    1, & \text{otherwise}.
\end{cases}\end{equation}
Similarly we can compute $M_{\cS, \chi}$:
\begin{align}\label{eq:M_S3}
    \nonumber M_{\cS, \chi} &= \prod_{p \in \Sigma_{\QQ} - \cS} L_p(1, \chi) \\
    &= \begin{cases}
    L_2(1,\chi) \prod_{p \mid \Delta} L_p(1,\chi),& \text{if 2 is unramified and } \Frob_2 = C_1,\\
    \prod_{p \mid \Delta} L_p(1,\chi), & \text{otherwise}.
    \end{cases}
\end{align}

From the definition of $k_\chi(p)$ in \eqref{eq:k}, we have
\[k_\chi(p) = \begin{cases}
    0, & \text{if } \chi(p) = 0 \text{ or } \chi(p) = 1, \\
    1/p, & \text{if } \chi(p) = -1, \\
    p/(p-1)^2, & \text{if } \chi(p) = 2.
\end{cases}\]
This allows us to produce $K(1,\chi)$ for $\chi = \chi_1, \chi_2$ according to \eqref{eq:K_P}:
\begin{align*}
    K(1, \chi_1) &= \prod_{\Frob_p = C_2} \left(1 - \frac{1}{p^2}\right)^{-1}\\
    K(1, \chi_2) &= \prod_{\Frob_p = C_1} \left(1 - \frac{1}{(p-1)^2}\right)^{-1} \prod_{\Frob_p = C_3} \left(1 - \frac{1}{p^2}\right)^{-1}
\end{align*}

It remains to describe $L$ and $R_{\cS, \chi}$. Since $\chi = \chi_1$ is one dimensional, we have $L_p(s,\chi) = M_p(s, \chi)$ and $R_{\cS, \chi_1} = 1$. On the other hand, $\chi_2$ is two dimensional, and as such $R_{\cS, \chi_2}$ is nontrivial. Thus, if 2 is not totally split $E/\QQ$, we may use \eqref{eq:B_S3} and \eqref{eq:M_S3} to give a more explicit description of $e^{-\gamma(E/\QQ, C)}$ given in \eqref{eq:constant}:
\begin{align*}
    e^{-\gamma(E/\QQ, C_1)} &= e^{-\gamma} \frac{\nrm \Delta}{\varphi(\Delta)} \left( \prod_{p \, \nmid \, \Delta} L_p(1,\chi_1) \prod_{\Frob_p = C_2} \left(1 - \frac{1}{p^2}\right)^{-1} \right) \times \\
    &\left( R_{\cS, \chi_2} \prod_{p \, \nmid \, \Delta}L_p(1, \chi_2)\inv
    \prod_{\Frob_p = C_1} \left(1 - \frac{1}{(p-1)^2}\right)^{-1} \prod_{\Frob_p = C_3} \left(1 - \frac{1}{p^2}\right)^{-1} \right)^2 ,\\
    e^{-\gamma(E/\QQ, C_2)} &= e^{-\gamma}\frac{\nrm \Delta}{\varphi(\Delta)} \prod_{p \, \nmid \, \Delta}L(1, \chi_1) \prod_{\Frob_p = C_2} \left(1 - \frac{1}{p^2}\right) ,\\
    e^{-\gamma(E/\QQ, C_3)} &= e^{-\gamma}\frac{\nrm \Delta}{\varphi(\Delta)} \left( \prod_{p \, \nmid \, \Delta} L_p(1,\chi_1)\inv \prod_{\Frob_p = C_2} \left(1 - \frac{1}{p^2}\right)^{-1} \right) \times\\ 
    &\left( \frac{\prod_{p \, \nmid \, \Delta}L_p(1, \chi_2)}{R_{\cS, \chi_2}} \prod_{\Frob_p = C_1} \left(1 - \frac{1}{(p-1)^2}\right) \prod_{\Frob_p = C_3} \left(1 - \frac{1}{p^2}\right) \right).
\end{align*}

If $2$ is unramified and totally split in $E$, these can be modified by taking $B_{\cS, \chi}$ and $M_{\cS, \chi}$ as in \eqref{eq:B_S3} and \eqref{eq:M_S3}.

We can use Theorem \ref{thm:main} \eqref{eq:altconstant} for an alternate determination of the constants $e^{-\gamma(E/\mathbb{Q}, C_i)}$ above. One finds,
\begin{align*}
    e^{-\gamma(E/\mathbb{Q}, C_1)} &= e^{-\gamma_F}\frac{\nrm(\Delta)}{\varphi(\Delta)}\prod_{\mathrm{Frob}_P=C_1}\left(1-\frac{1}{\nrm P}\right)^{5}\prod_{\mathrm{Frob}_P\neq C_1}\left(1-\frac{1}{\nrm P}\right)^{-1},\\
 e^{-\gamma(E/\mathbb{Q}, C_2)} &= e^{-\gamma_F}\frac{\nrm(\Delta)}{\varphi(\Delta)}\prod_{\mathrm{Frob}_P=C_2}\left(1-\frac{1}{\nrm P}\right)^{1}\prod_{\mathrm{Frob}_P\neq C_2}\left(1-\frac{1}{\nrm P}\right)^{-1},\\
     e^{-\gamma(E/\mathbb{Q}, C_3)} &= e^{-\gamma_F}\frac{\nrm(\Delta)}{\varphi(\Delta)}\prod_{\mathrm{Frob}_P=C_3}\left(1-\frac{1}{\nrm P}\right)^{2}\prod_{\mathrm{Frob}_P\neq C_3}\left(1-\frac{1}{\nrm P}\right)^{-1} .\\
\end{align*}

\subsection{Future work}
We suspect our methods can be extended to the case of global function fields in a straightforward manner. More generally, it would be interesting to consider the case of varieties over finite fields, by using Lebaque's \cite{lebacque2007} generalization of Mertens' theorem in place of Rosen's theorem (Theorem \ref{thm:RosenThm2}).
\begin{theorem}[Theorem 5 in \cite{lebacque2007}]
    Let $X$ be a smooth, projective, and geometrically irreducible variety of dimension $d$ defined over a finite field $\Fq$. Call $\varkappa_X$ the residue of the Weil zeta function $\zeta_X(s)$ at $s=d$. Then
    \begin{equation}
        \prod_{\deg P \, \leq \, N} \paren{1-\frac{1}{(\nrm P)^d}} = \frac{e^{-\gamma_X}}{N} + O\paren{\frac{1}{N^2}},
    \end{equation}
    where the product runs over the closed points $P \in X$ and $\gamma_X = \gamma + \log(\varkappa_X \log q)$.
\end{theorem}

\section{Acknowledgements}
We would like to thank Robert Lemke Oliver and David Zureick-Brown for helpful conversations and Paul Pollack bringing the work of Languasco and Zaccagnini to our attention. We also thank Kenneth Williams for suggesting we investigate the case of primes represented by quadratic forms.

\bibliography{NFChebotarevMertensSAPDKCK.bib}{} \bibliographystyle{amsalpha}

\end{document}